\documentclass[11pt]{article}
\usepackage{amsthm,geometry,amssymb,amsmath,cite,setspace,hyperref}
\newtheorem{theorem}{Theorem}
\newtheorem{lemma}[theorem]{Lemma}
\newtheorem{proposition}[theorem]{Proposition}

\allowdisplaybreaks

\usepackage{lineno}

\title{Biholes in balanced bipartite graphs}
\author{Stefan Ehard, Elena Mohr, and Dieter Rautenbach\\[3mm]
\normalsize Institute of Optimization and Operations Research, Ulm University, Germany\\
\texttt{$\{$stefan.ehard,elena.mohr,dieter.rautenbach$\}$@uni-ulm.de}}
\date{}

\begin{document}
\maketitle
\onehalfspace

\begin{abstract}
A bihole in a bipartite graph $G$ with partite sets $A$ and $B$ 
is an independent set $I$ in $G$ with $|I\cap A|=|I\cap B|$.
We prove lower bounds on the largest order of biholes 
in balanced bipartite graphs
subject to conditions involving the vertex degrees and the average degree.\\[2mm]
\end{abstract}

\section{Introduction}

In \cite{axtowe,axsesnwe} 
Axenovich et al.~study {\it biholes} 
defined as independent sets in bipartite graphs
containing equally many vertices from 
both parts of a fixed bipartition.
They present several lower bounds 
on the order of largest biholes 
subject to degree conditions.
Here we pursue some of the questions
motivated by \cite{axsesnwe}.
For a detailed discussion of the motivation of biholes,
we refer to \cite{axsesnwe}.
First, we collect some notation and definitions.
We consider only finite, simple, and undirected graphs.
For a graph $G$, we denote 
the vertex set,
the edge set,
the order, and
the size by
$V(G)$,
$E(G)$,
$n(G)$, and
$m(G)$, 
respectively.
Let $G$ be a bipartite graph with partite sets $A$ and $B$.
A {\it bihole of order $k$ in $G$}
is an independent set $I$ in $G$
with $|I\cap A|=|I\cap B|=k$.
Note that the definition of a bihole tacitly 
requires to fix a bipartition of $G$,
which is unique only if $G$ is connected.
Note furthermore that the {\it order} of a bihole $I$ 
is half the cardinality of the set $I$.
Let $\tilde{\alpha}(G)$ be the largest order 
of a bihole in $G$.
A bipartite graph $G$ with partite sets $A$ and $B$
is {\it balanced} if $|A|=|B|$.
For an integer $k$, let $[k]$ be the set of positive integers at most $k$,
and let $[k]_0=\{ 0\}\cup [k]$.

For positive integers $n$ and $\Delta$,
Axenovich et al.~\cite{axsesnwe} define 
$f(n,\Delta)$ as the largest integer $k$ 
such that every bipartite graph $G$
with partite sets $A$ and $B$ satisfying
\begin{itemize}
\item $|A|=|B|=n$, and 
\item the degree $d_G(u)$ of every vertex $u$ from $A$ is at most $\Delta$,
\end{itemize}
has a bihole of order $k$.
Similarly, they define $f^*(n,\Delta)$ as the largest integer $k$ 
such that every balanced bipartite graph $G$ 
of order $2n$ and maximum degree at most $\Delta$,
has a bihole of order $k$.
The definitions immediately imply $f(n,\Delta)\leq f^*(n,\Delta)$.

In \cite{axsesnwe} Axenovich et al.~show the following results
for integers $n$ and $\Delta$ with $n\geq\Delta\geq 2$:
\begin{eqnarray}
f(n,2)&=&\left\lceil\frac{n}{2}\right\rceil-1,\label{eax1}\\
f(n,\Delta)&\geq&\left\lfloor\frac{n-2}{\Delta}\right\rfloor,\label{eax2}\\
f(n,\Delta)&=&\Theta\left(\frac{\ln \Delta}{\Delta}n\right)
\mbox{ for large but fixed $\Delta$ and $n$ sufficiently large, and}\label{eax3}\\
0.3411n&<&f(n,3)\leq f^*(n,3)<0.4591n
\mbox{ for $n$ sufficiently large.}\label{eax4}
\end{eqnarray}
They explicitly ask for the value of $f(n,3)$ for sufficiently large $n$.

While the parameters $f(n,\Delta)$ and $f^*(n,\Delta)$ 
might appear closely related to the independence number
$\alpha(G)$ of a graph $G$,
and one might be tempted to expect a similar behavior,
the requirement to contain equally many vertices
from both partite sets imposes a strict condition.
In fact, balancing the intersections with the partite sets
seems to be one of the challenges in proofs 
about these parameters.

The following three tight lower bounds on the independence number 
$\alpha(G)$ of a graph $G$ 
with average degree $d=\frac{2m(G)}{n(G)}$ 
and maximum degree at most $\Delta$ 
are well known \cite{ca,we,tu}:
\begin{eqnarray}\label{ealpha}
\alpha(G)\geq \sum\limits_{u\in V(G)}\frac{1}{d_G(u)+1}
\geq \frac{n}{d+1}\geq \frac{n}{\Delta+1}.
\end{eqnarray}
The inequality (\ref{eax2}) translates
the final bound in (\ref{ealpha}) from independent sets to biholes,
but (\ref{eax3}) indicates that asymptotically
stronger lower bounds hold.
The result (\ref{eax3}) implies the following similar result involving the average degree.

\begin{proposition}\label{propasymp}
There exists a real $d_0$ such that, 
for every real $d\geq d_0$,
there is some integer $n_0(d)$ such that, 
for every integer $n\geq n_0(d)$,
the following statement holds:
If $G$ is a balanced bipartite graph 
of order $2n$ that has at most $dn$ edges, then 
$\tilde{\alpha}(G)\geq \frac{\ln(d)}{8d}n$.
\end{proposition}
\begin{proof}
Axenovich et al.~\cite{axsesnwe} show the following:
\begin{quote}
{\it There exists an integer $\Delta_1$ such that, 
for every integer $\Delta\geq \Delta_1$,
there is some integer $n_1(\Delta)$ such that 
$f(n,\Delta)\geq \frac{\ln(\Delta)n}{2\Delta}$
for every $n\geq n_1(\Delta)$.}
\end{quote}
Let $d_0=\max\left\{1,\frac{\Delta_1+1}{2}\right\}$ and,
for every real $d\geq d_0$, let $n_0(d)=2n_1(\lfloor 2d\rfloor)$.
Now, let $d$ be any real at least $d_0$,
and let $n$ be any integer at least $n_0(d)$.
Let $G$ be a balanced bipartite graph 
of order $2n$ that has at most $dn$ edges.
Let $A$ and $B$ be the partite sets of $G$.
Let $n_{>2d}$ be the number of vertices $u$ in $A$
with $d_G(u)>2d$.
Since $2dn_{>2d}\leq dn$,
we have $n_{>2d}\leq \frac{n}{2}$,
which implies that 
$A'=\{ u\in A:d_G(u)\leq \lfloor 2d\rfloor\}$
contains at least $\frac{n}{2}$ vertices.
Let $G'$ arise from $G$ by removing 
all vertices in $A\setminus A'$ from $A$
as well as any $|A\setminus A'|$ vertices from $B$.
Clearly, the graph $G'$ is a balanced bipartite graph
of order $2n'$ with $n'\geq \frac{n}{2}$
such that $d_{G'}(u)\leq \lfloor 2d\rfloor$ for every vertex in $A'$.
Since $d\geq 1$,
$\lfloor 2d\rfloor\geq \Delta_1$,
and $n'\geq \frac{n_0(d)}{2}=n_1(\lfloor 2d\rfloor)$,
the result from \cite{axsesnwe} implies
$$\tilde{\alpha}(G)
\geq \tilde{\alpha}(G')
\geq f(n',\lfloor 2d\rfloor)
\geq \frac{\ln(\lfloor 2d\rfloor)n'}{2\lfloor 2d\rfloor}
\geq \frac{\ln(d)n}{8d}.$$
\end{proof}
\noindent Inspired by the second bound in (\ref{ealpha}),
we prove the following,
which, in view of Proposition \ref{propasymp},
is interesting for small values of $d$ or $n$.

\begin{theorem}\label{theoremx}
If $G$ is a balanced bipartite graph 
of order $2n$ that has at most $dn$ edges
for some non-negative real $d$, then 
\begin{eqnarray}\label{ex}
\tilde{\alpha}(G)\geq \frac{n}{d+1}-2.
\end{eqnarray}
\end{theorem}
\noindent Furthermore, 
we contribute a small improvement of the lower bound on $f(n,3)$ from \cite{axsesnwe}.
Therefore, we need the following refined version of $f(n,\Delta)$:
For non-negative integers $d_1<d_2<\ldots<d_\ell$ 
and $n_1,n_2,\ldots,n_\ell$,
let $\tilde{\alpha}(d_1^{n_1},d_2^{n_2},\ldots,d_\ell^{n_\ell})$ be the largest $k$ 
such that every bipartite graph $G$
with partite sets $A$ and $B$ such that 
\begin{itemize}
\item $|A|=|B|=n_1+n_2+\cdots+n_\ell$, and
\item $n_i=|\{ u\in A:d_G(u)=d_i\}|$ for every $i\in [\ell]$,
\end{itemize}
has a bihole of order $k$.
For the considered graphs,
the sequence  
$\underbrace{d_1\ldots d_1}_{n_1}
\ldots
\underbrace{d_\ell\ldots d_\ell}_{n_\ell}$
is the degree sequence of the vertices in $A$.
Note that 
$$f(n,\Delta)=\min\Big\{ \tilde{\alpha}(0^{n_0},\ldots,\Delta^{n_{\Delta}}):
n_0,\ldots,n_{\Delta}\in \mathbb{N}_0\mbox{ with }
n=n_0+\cdots +n_{\Delta}\Big\}.$$
Our next result can be considered to be a refinement of (\ref{eax1}).

\begin{theorem}\label{lemma1}
$\tilde{\alpha}(0^{n_0},1^{n_1},2^{n_2})\geq \frac{3}{4}n_0+\frac{1}{2}(n_1+n_2)-\frac{7}{4}.$
\end{theorem}
\noindent Finally, combining Theorem \ref{lemma1}
with the approach of Axenovich et al.~\cite{axsesnwe},
allows to slightly improve their lower bound on $f(n,3)$
as follows.
\begin{theorem}\label{theorem1}
For every $\epsilon\geq0$, 
there is some $n_0$ 
such that $f(n,3)\geq \big(0.34917-\epsilon\big)n$ 
for every $n\geq n_0$.
\end{theorem}
\noindent The proofs of the stated results as well as of further auxiliary statements are given in the following section.

\section{Proofs}

We begin with a restricted analogue of Theorem \ref{lemma1}.

\begin{lemma}\label{lemma3}
$\tilde{\alpha}(0^{n_0},1^{n_1})\geq n_0+\frac{n_1}{2}-\frac{1}{2}$.
\end{lemma}
\begin{proof}
Let $G$ be a bipartite graph
with partite sets $A$ and $B$ such that 
\begin{itemize}
\item $|A|=|B|=n_0+n_1$, and
\item $n_i=|\{ u\in A:d_G(u)=i\}|$ for $i\in \{ 0,1\}$.
\end{itemize}
Let $G_1,\ldots,G_k$ be the components of $G$ 
that are of order more than $2$.
Each $G_i$ is a star with $n_{1,i}\geq 2$ endvertices from $A$
and a center vertex from $B$.
It follows that $G$ contains $\ell=n_1-\sum\limits_{i=1}^kn_{1,i}$
components that are $K_2$'s, and $B$ contains 
$n_0+\sum\limits_{i=1}^k(n_{1,i}-1)$ isolated vertices.
Now, there is a bihole $I$ in $G$ containing
\begin{itemize}
\item all $n_0$ isolated vertices from $A$,
\item at least $\frac{\ell-1}{2}$ vertices of degree $1$ from $A$
as well as at least $\frac{\ell-1}{2}$ vertices of degree $1$ from $B$; all coming from $K_2$ components,
\item $\sum\limits_{i=1}^k(n_{1,i}-1)$ vertices of degree $1$ from $A$; coming from the $G_i$'s, and
\item all $n_0+\sum\limits_{i=1}^k(n_{1,i}-1)$ isolated vertices  from $B$.
\end{itemize}
Since $k\leq \frac{n_1-\ell}{2}$, we obtain that $I$ has order at least
\begin{eqnarray*}
n_0+\frac{\ell-1}{2}+\sum\limits_{i=1}^k(n_{1,i}-1)
&=&n_0+\frac{\ell-1}{2}+(n_1-\ell-k)\\
&\geq &n_0+n_1-\frac{\ell}{2}-\frac{n_1-\ell}{2}-\frac{1}{2}\\
&=&n_0+\frac{n_1}{2}-\frac{1}{2},
\end{eqnarray*}
which completes the proof.
\end{proof}
\noindent Now, we proceed to the proof of Theorem \ref{theoremx}.

\begin{proof}[Proof of Theorem \ref{theoremx}.]
Suppose, for a contradiction, 
that $G$ is a counterexample of minimum order $2n$.
Let $\Delta_A$ be the maximum degree of the vertices in $A$.
First, we assume that $\Delta_A<2$.
Let $A$ contain $n_i$ vertices of degree $i$ for $i\in \{ 0,1\}$.
Since $G$ has $n_1$ edges, 
we have $d\geq \frac{n_1}{n_0+n_1}$.
Now, since 
$n_0+\frac{n_1}{2}
\geq \frac{n_0+n_1}{\frac{n_1}{n_0+n_1}+1}$, 
Lemma \ref{lemma3} implies 
$$\tilde{\alpha}(G)
\geq n_0+\frac{n_1}{2}-\frac{1}{2}
\geq \frac{n_0+n_1}{\frac{n_1}{n_0+n_1}+1}-\frac{1}{2}
=\frac{n}{d+1}-\frac{1}{2},$$
and (\ref{ex}) follows.
Next, we assume that $2\leq \Delta_A<d+1$.
In this case, the inequality (\ref{eax2}) implies
$$\tilde{\alpha}(G)\geq \left\lfloor\frac{n-2}{\Delta_A}\right\rfloor
>\frac{n}{\Delta_A}-2
>\frac{n}{d+1}-2.$$
Finally, we may assume that $\Delta_A\geq d+1$.
By symmetry, we may also assume that 
the maximum degree $\Delta_B$ of the vertices in $B$
satisfies $\Delta_B\geq d+1$. 
Let $u\in A$ and $v\in B$ be vertices of degree at least $d+1$. 
The graph $G'=G-\{u,v\}$ is balanced with partite sets of order $n-1$, 
and at most $nd-d_G(u)-d_G(v)+1\leq nd-2d-1$ edges. 
By the choice of $G$, the graph $G'$ is no counterexample, 
and we obtain 
$$\tilde{\alpha}(G)\geq  \tilde{\alpha}(G')\geq \frac{n-1}{\frac{nd-2d-1}{n-1}+1}-2
=\frac{(n-1)^2}{(d+1)(n-2)}-2
\geq \frac{n}{d+1}-2,$$
where we use $(n-1)^2\geq n(n-2)$.
This completes the proof.
\end{proof}
\noindent The following result illustrates a different approach for $d=2$,
and gives a better additive constant.

\begin{proposition}\label{proposition1}
If $G$ is a balanced bipartite graph 
of order $2n\geq 4$ that has at most $2n$ edges,
then $\tilde{\alpha}(G)\geq \frac{n-2}{3}$.
\end{proposition}
\begin{proof}
We prove the statement by induction on $n$.
Let $A$ and $B$ be the partite sets of $G$.
For $n=2$, the statement is trivial.
Now, let $n\geq 3$.
Let $\delta_A=\min\{ d_G(u):u\in A\}$,
$\Delta_A=\max\{ d_G(u):u\in A\}$,
and define $\delta_B$ as well as $\Delta_B$ analogously.
By the result (\ref{eax1}) of Axenovich et al.~\cite{axsesnwe},
and, since $\frac{n}{2}-1\geq \frac{n-2}{3}$,
we may assume that $\Delta_A,\Delta_B\geq 3$.
Since $G$ has at most $2n$ edges, 
this implies $\delta_A,\delta_B\leq 1$.

First, suppose that $\delta_A=0$.
Let $u$ be an isolated vertex from $A$.
Let $v$ be a vertex of degree $\delta_B$ from $B$.
Let $u'$ be a vertex from $A\setminus \{ u\}$
of largest possible degree such that $N_G(v)\subseteq \{ u'\}$.
Let $v'$ be a vertex on degree $\Delta_B$ from $B$.
Note that $m(G')\leq 2(n-2)$.
By induction, the graph $G'=G-\{ u,v,u',v'\}$
has a bihole $I'$ of order at least $\frac{n-2-2}{3}$.
Since adding $u$ and $v$ to $I'$ yields a bihole in $G$,
the desired statement follows.
Hence, by symmetry, we may assume that $\delta_A=\delta_B=1$.

Next, suppose that there are non-adjacent vertices 
$u$ from $A$ 
and
$v$ from $B$
that are both of degree $1$.
Let $v'$ be the neighbor of $u$,
and 
let $u'$ be the neighbor of $v$.
If $d_G(u')\geq 3$,
then, by induction,
the graph $G'=G-\{ u,v,u',v'\}$
has a bihole $I'$ of order at least $\frac{n-2-2}{3}$.
Adding $u$ and $v$ to $I'$ yields a bihole of $G$
of the desired order.
Hence, by symmetry, we may assume that 
$d_G(u'),d_G(v')\leq 2$.
Let $u''$ be a vertex of degree $\Delta_A$ from $A$,
and 
let $v''$ be a vertex of degree $\Delta_B$ from $B$.
Let $G''$ be the graph $G-\{ u,v,u',v',u'',v''\}$.
It is easy to see that $m(G'')\leq 2(n-3)$.
By induction,
the graph $G''$
has a bihole $I''$ of order at least $\frac{n-3-2}{3}$.
Adding $u$ and $v$ to $I''$ yields a bihole of $G$
of the desired order.
Hence, we may assume that $A$ and $B$ 
both contain unique vertices of degree $1$, 
say $u$ and $v$, respectively, and that $u$ and $v$ are adjacent.
Since $n\geq 3$, $m(G)\leq 2n$, and $\Delta_A\geq 3$,
there is a vertex $u'$ of degree $2$ in $A$.
Let $v'$ and $v''$ be the two neighbors of $u'$.
Let $u''$ be a vertex of degree $\Delta_A$ from $A$.
Let $G''$ be the graph $G-\{ u,v,u',v',u'',v''\}$.
It is easy to see that $m(G'')\leq 2(n-3)$.
By induction,
the graph $G''$
has a bihole $I''$ of order at least $\frac{n-3-2}{3}$.
Adding $v$ and $u'$ to $I''$ yields a bihole of $G$
of the desired order,
which completes the proof.
\end{proof}
\noindent Our next goal is the proof of Theorem \ref{lemma1},
which refines (\ref{eax1}), and allows to slightly improve the lower bound on $f(n,3)$.

\begin{lemma}\label{lemma2}
If $G$ is a connected bipartite graph with partite sets $A$ and $B$ such that 
$|A|<|B|$ and every vertex in $A$ has degree at most $2$, 
then $G$ is a tree, $|B|=|A|+1$, and every vertex in $A$ has degree exactly $2$.
Furthermore, for every $i$ in $[|A|]_0$,
there is an independent set $I$ in $G$ with 
$|I\cap A|=i$ and $|I\cap B|=|A|-i$.
\end{lemma}
\begin{proof}
Since every vertex in $A$ has degree at most $2$, 
the graph $G$ has at most $2|A|$ edges.
Since $G$ is connected, it has at least $|A|+|B|-1\geq 2|A|$ edges.
It follows that $G$ has exactly $2|A|$ edges, 
every vertex in $A$ has degree exactly $2$,
$|B|=|A|+1$,
and $G$ is a tree.

We prove the existence of the desired independent sets by induction on $|A|$.
For $|A|=1$, the statement is trivial.
Now, let $|A|\geq 2$.
Clearly, choosing $I$ as $A$ yields $|I\cap A|=|A|$ and $|I\cap B|=|A|-|A|=0$,
that is, the statement is trivial for $i=|A|$.
Now, let $i\in [|A|-1]_0$.
Let $u$ be a vertex of degree $1$, and let $v$ be its unique neighbor.
By induction applied to $G'=G-\{ u,v\}$,
the graph $G'$ has an independent set $I'$ with 
$|I'\cap A|=i$ and $|I'\cap B|=(|A|-1)-i$, and adding $u$ to $I'$ 
yields the desired independent set.
\end{proof}

\begin{proof}[Proof of Theorem \ref{lemma1}]
By induction on $n_0$, we show that every bipartite graph $G$
with partite sets $A$ and $B$ such that 
\begin{itemize}
\item $|A|=|B|=n_0+n_1+n_2$, and
\item $n_i=|\{ u\in A:d_G(u)=i\}|$ for every $i\in [2]_0$,
\end{itemize}
has a bihole of order at least 
$\frac{3}{4}n_0+\frac{1}{2}(n_1+n_2)-\frac{7}{4}$.
If $n_0\leq 3$, then $G$ has a bihole of order at least 
\begin{eqnarray*}
\tilde{\alpha}(0^{n_0},1^{n_1},2^{n_2})\geq 
\tilde{\alpha}(0^{0},1^{0},2^{n_0+n_1+n_2})
\stackrel{(\ref{eax1})}{\geq}
\frac{n_0+n_1+n_2}{2}-1
\geq \frac{3}{4}n_0+\frac{1}{2}(n_1+n_2)-\frac{7}{4}.
\end{eqnarray*}
Now, let $n_0\geq 4$.
Let $u_1,\ldots, u_4$ be four isolated vertices from $A$.
Let $G_1,\ldots,G_r$ be the components of $G$
with $|V(G_i)\cap A|<|V(G_i)\cap B|$.
Since $|A|=|B|$ and $n_0\geq 4$, 
there is at least one such component, that is, we have $r\geq 1$.
By Lemma \ref{lemma2},
each $G_i$ is a tree with $|V(G_i)\cap B|=|V(G_i)\cap A|+1$,
which implies $r\geq n_0\geq 4$.
Let 
$A_i=V(G_i)\cap A$,
$B_i=V(G_i)\cap B$,
and
$a_i=|A_i|$,
for $i$ in $[4]$.
Clearly, we may assume that $a_1\leq a_2\leq a_3\leq a_4$.
If $B$ contains an isolated vertex $v$, then, applying induction to $G'=G-\{ u_1,v\}$,
we obtain that $G'$ has a bihole of order at least 
$\frac{3}{4}(n_0-1)+\frac{1}{2}(n_1+n_2)-\frac{7}{4}$,
and adding $u_1$ and $v$ yields a bihole of more than the desired order.
Hence, we may assume that no vertex in $B$ is isolated, 
in particular, we have $a_1\geq 1$.

First, we assume that $a_1$ and $a_2$ have different parities modulo $2$.
By Lemma \ref{lemma2}, there is an independent set $I_2$ of $G_2$ with
$$|I_2\cap A|=\frac{a_2+a_1-1}{2}\,\,\,\,\,\,\,\,\mbox{ and }\,\,\,\,\,\,\,\,
|I_2\cap B|=\frac{a_2-a_1+1}{2}.$$
By induction, the graph $G'=G-\big(\{ u_1,u_2\}\cup V(G_1)\cup V(G_2)\big)$
has a bihole $I'$ of order at least 
$\frac{3}{4}(n_0-2)+\frac{1}{2}(n_1+n_2-a_1-a_2)-\frac{7}{4}.$
Now, the set $\big(\{ u_1,u_2\}\cup B_1\cup I_2\big)\cup I'$ 
is a bihole in $G$ of order at least
$$\frac{1}{2}\big(2+(a_1+1)+a_2\big)+
\left(\frac{3}{4}(n_0-2)+\frac{1}{2}(n_1+n_2-a_1-a_2)-\frac{7}{4}\right)
=\frac{3}{4}n_0+\frac{1}{2}(n_1+n_2)-\frac{7}{4}.$$
Hence, we may assume that 
$a_1$ and $a_2$ have the same parity modulo $2$, and, by symmetry, 
that also $a_3$ and $a_4$ have the same parity modulo $2$.
Note that $\frac{a_4+a_3-2}{2}\in [|A|]_0$.
By Lemma \ref{lemma2}, there is an independent set $I_2$ of $G_2$ with
$$|I_2\cap A|=\frac{a_2+a_1}{2}\,\,\,\,\,\,\,\,\mbox{ and }\,\,\,\,\,\,\,\,
|I_2\cap B|=\frac{a_2-a_1}{2},$$
as well as an independent set $I_4$ of $G_4$ with
$$|I_4\cap A|=\frac{a_4+a_3-2}{2}\,\,\,\,\,\,\,\,\mbox{ and }\,\,\,\,\,\,\,\,
|I_4\cap B|=\frac{a_4-a_3+2}{2}.$$
By induction, the graph $G''=G-\big(\{ u_1,u_2,u_3,u_4\}\cup V(G_1)\cup V(G_2)\cup V(G_3)\cup V(G_4)\big)$
has a bihole $I''$ of order at least 
$\frac{3}{4}(n_0-4)+\frac{1}{2}(n_1+n_2-a_1-a_2-a_3-a_4)-\frac{7}{4}.$
Now, the set $\big(\{ u_1,u_2,u_3,u_4\}\cup B_1\cup I_2\cup B_3\cup I_4\big)\cup I''$ 
is a bihole in $G$ of order at least
\begin{eqnarray*}
&& \frac{1}{2}\big(4+(a_1+1)+a_2+(a_3+1)+a_4\big)+
\left(\frac{3}{4}(n_0-4)+\frac{1}{2}(n_1+n_2-a_1-a_2-a_3-a_4)-\frac{7}{4}\right)\\
&=&\frac{3}{4}n_0+\frac{1}{2}(n_1+n_2)-\frac{7}{4},
\end{eqnarray*}
which completes the proof.
\end{proof}
\noindent The following example shows that 
the coefficient $\frac{3}{4}$ for $n_0$ 
in Theorem \ref{lemma1} is best possible:
For an even integer $i$, 
let the bipartite graph $G$
have partite sets $A$ and $B$ 
and exactly $2i$ components 
such that 
there are $i$ isolated vertices that all belong to $A$, 
and $i$ paths $P_1,\ldots,P_i$, each of order $4i+1$, 
whose endpoints all belong to $B$. 
Note that $|A|=|B|=i+2i^2$, $n_0=i$, and $n_2=2i^2$.
Let $I$ be a largest bihole in $G$. 
From every path $P_i$,  
at most $2i+1$ vertices can belong to $I$, 
and, if $2i+1$ vertices belong to $I$, 
then $V(P_i)\cap I\subseteq B$.
If in more than $\frac{i}{2}$ of the paths $P_i$,
at least $2i+1$ vertices belong to $I$, then
$$|I\cap A|\leq \left( \frac{i}{2}-1\right) 2i+i=i^2-i
<i^2+\frac{5}{2}i+1=\left(\frac{i}{2}+1\right)(2i+1)
\leq|I\cap B|,$$
which is a contradiction. 
Hence, in at most $\frac{i}{2}$ of the paths $P_i$,
at least $2i+1$ vertices belong to $I$, which implies
$$|I|\leq\frac{1}{2}\left(i+(2i+1)\frac{i}{2}+2i\frac{i}{2}\right)
=i^2+\frac{3}{4}i=\frac{3}{4} n_0+\frac{1}{2}n_2.$$
It seems a challenging problem to determine the value $\tilde{\alpha}(0^{n_0},1^{n_1},2^{n_2})$
exactly for all choices of $n_0$, $n_1$, and $n_2$.
In fact, depending on the relative values of the $n_i$,
they should contribute to this value with different coefficients.
If, for instance, $n_2=0$, then, by Lemma \ref{lemma3},
the coefficient of $n_0$ is $1$ rather than 
$\frac{3}{4}$ as in Theorem \ref{lemma1}.

For the next proof, 
we need the following {\it Simple Concentration Bound} \cite{more}:
\begin{quote}
{\it Let $X$ be a random variable determined 
by $n$ independent trials $T_1,\ldots,T_n$
such that changing the outcome of any one trial 
can affect $X$ by at most $c$, then 
\begin{eqnarray}\label{econc}
\mathbb{P}\Big[|X-\mathbb{E}[X]|>t\Big]\leq 2e^{-\frac{t^2}{2c^2n}}
\,\,\,\,\,\,\,\mbox{ for every $t>0$.}
\end{eqnarray}}
\end{quote}

\begin{proof}[Proof of Theorem \ref{theorem1}]
Let $G$ be a bipartite graph with partite sets $A$ and $B$ such that 
$|A|=|B|=n$ and every vertex in $A$ has degree at most $3$.
We need to show that $G$ has a bihole of order at least $\big(0.34917-o(n)\big)n$.
Therefore, let $\epsilon$ be such that $0<\epsilon<\frac{1}{2\ln(8)}<0.25$.
Let $B_{\rm large}$ be the set of vertices in $B$ 
of degree more than $\epsilon^{3/2}\sqrt{n}$,
and let $B_{\rm small}=B\setminus B_{\rm large}$.
Since $G$ has at most $3n$ edges, 
we have $|B_{\rm large}|\leq \frac{3\sqrt{n}}{\epsilon^{3/2}}$.
Let $G^{(1)}$ arise from $G$ by removing $B_{\rm large}$
as well as any set of $|B_{\rm large}|$ vertices from $A$.
Let 
$$n_i^{(1)}=\left|\left\{ u\in V(G^{(1)})\cap A:d_{G^{(1)}}(u)=i\right\}\right|$$ 
for $i\in [3]_0$, and let $n^{(1)}=\left|V(G^{(1)})\cap A\right|$, that is,
$$n^{(1)}=n_0^{(1)}+n_1^{(1)}+n_2^{(1)}+n_3^{(1)}
=n-|B_{\rm large}|\geq \left(1-\frac{3}{\epsilon^{3/2}\sqrt{n}}\right)n.$$
Let $p$ be the real solution of the equation $p=(1-p)^3$,
that is, $p\approx 0.31767$.
Let $B^{(1)}$ be a random subset of $B_{\rm small}$
that arises by adding each of the $n^{(1)}$ vertices in $B_{\rm small}$ to the set $B^{(1)}$
independently at random with probability $p$.
Let $G^{(2)}$ arise from $G^{(1)}$ by removing $B^{(1)}$, 
let $b^{(1)}=|B^{(1)}|$, and 
let 
$$n_i^{(2)}=\left|\left\{ u\in V(G^{(2)})\cap A:d_{G^{(2)}}(u)=i\right\}\right|$$ 
for $i\in [3]_0$.

For the random variables $b^{(1)}$, $n_0^{(2)}$, and $n_3^{(2)}$, we obtain
\begin{eqnarray*}
\mathbb{E}\left[b^{(1)}\right] & = & pn^{(1)},\\
\mathbb{E}\left[n_0^{(2)}\right] & = & n^{(1)}_0+pn_1^{(1)}+p^2n_2^{(1)}+p^3n_3^{(1)}\geq p^3n^{(1)},\\
\mathbb{E}\left[n_3^{(2)}\right] & = & (1-p)^3n^{(1)}_3\leq (1-p)^3n^{(1)}.
\end{eqnarray*}
Applying (\ref{econc}) with $c=\epsilon^{3/2}\sqrt{n}$ in each case,
and using $\epsilon<\frac{1}{2\ln(8)}$, we obtain
\begin{eqnarray*}
\mathbb{P}\Big[\left|b^{(1)}-\mathbb{E}\left[b^{(1)}\right]\right|>\epsilon n^{(1)}\Big] & \leq &
2e^{-\frac{\left(\epsilon n^{(1)}\right)^2}{2\epsilon^3n n^{(1)}}}
\leq 2e^{-\frac{\left(1-\frac{3}{\epsilon^{3/2}\sqrt{n}}\right)}{2\epsilon}}<\frac{1}{3},\\
\mathbb{P}\Big[\left|n_0^{(2)}-\mathbb{E}\left[n_0^{(2)}\right]\right|>\epsilon n^{(1)}\Big] & < & \frac{1}{3},\mbox{ and}\\
\mathbb{P}\Big[\left|n_3^{(2)}-\mathbb{E}\left[n_3^{(2)}\right]\right|>\epsilon n^{(1)}\Big] & < & \frac{1}{3},
\end{eqnarray*}
for $n$ sufficiently large.

For $n$ sufficiently large, 
the union bound implies the existence of a choice of $B^{(1)}$ such that 
\begin{eqnarray*}
b^{(1)} & \leq & (p+\epsilon)n^{(1)},\\
n_0^{(2)}& \geq & (p^3-\epsilon)n^{(1)},\mbox{ and}\\
n_3^{(2)}& \leq & \big((1-p)^3+\epsilon\big)n^{(1)}=(p+\epsilon)n^{(1)}.
\end{eqnarray*}
Let $G^{(3)}$ arise from $G^{(2)}$ by removing 
\begin{itemize}
\item a set 
containing $\max\left\{ b^{(1)},n_3^{(2)}\right\}$ vertices from $V(G^{(2)})\cap A$ 
including all vertices from $V(G^{(2)})\cap A$ that are of degree $3$ in $G^{(2)}$ and as few isolated vertices of $G^{(2)}$ as possible, and
\item a set containing $\max\left\{ b^{(1)},n_3^{(2)}\right\}$ vertices from $V(G^{(2)})\cap B$.
\end{itemize}
By construction all vertices in $V(G^{(3)})\cap A$ have degree at most $2$ in $G^{(3)}$.
Since 
$$(p^3-\epsilon)n^{(1)}+\max\left\{ b^{(1)},n_3^{(2)}\right\}\leq 
(p^3-\epsilon)n^{(1)}+(p+\epsilon)n^{(1)}
\leq 0.34974n^{(1)}\leq n^{(1)},$$
the number $n^{(3)}_0$ of vertices in $V(G^{(2)})\cap A$ 
that are isolated in $G^{(3)}$
is at least $(p^3-\epsilon)n^{(1)}$.
By Theorem \ref{lemma1}, the graph $G^{(3)}$, 
and, hence, also $G$, contains a bihole of order at least
\begin{eqnarray*}
&& \frac{3}{4}n_0^{(3)}+\frac{1}{2}\Big(n^{(1)}-n_0^{(3)}-\max\left\{ b^{(1)},n_3^{(2)}\right\}\Big)-C\\
& \geq & \frac{3}{4}(p^3-\epsilon)n^{(1)}+\frac{1}{2}\Big(n^{(1)}-(p^3-\epsilon)n^{(1)}-(p+\epsilon)n^{(1)}\Big)-C\\
& \geq & \left(\frac{3}{4}(p^3-\epsilon)+\frac{1}{2}\Big(1-p^3-p\Big)\right)\left(1-\frac{3}{\epsilon^{3/2}\sqrt{n}}\right)n-C\\
& \geq & \left(0.34917-\frac{3}{4}\epsilon\right)\left(1-\frac{3}{\epsilon^{3/2}\sqrt{n}}\right)n-C,
\end{eqnarray*}
which completes the proof. 
\end{proof}

\end{document}